\documentclass[11pt,reqno]{amsart}

\usepackage{enumerate, amsmath, amsthm, amsfonts, amssymb, 
mathrsfs, graphicx, paralist, ulem}
\usepackage[usenames, dvipsnames]{color}
\usepackage[margin=1in]{geometry} 
\usepackage{hyperref}
\usepackage{comment}
\usepackage{bbm}

\usepackage{mathabx} 
\usepackage{breqn} 
\usepackage{tikz-cd}
\usepackage{amsrefs}
\usepackage{ytableau}

\numberwithin{equation}{section}
\newtheorem{Theorem}[equation]{Theorem}
\newtheorem{Proposition}[equation]{Proposition}
\newtheorem{Lemma}[equation]{Lemma}
\newtheorem{Corollary}[equation]{Corollary}

\newtheorem{Problem}[equation]{Problem}
\newtheorem{Remark}[equation]{Remark}

\theoremstyle{Definition}
\newtheorem{Acknowledgement}[equation]{Acknowledgement}

\newtheorem{eg}[equation]{Example}

\newtheorem{Definition}[equation]{Definition}

\newtheoremstyle{named}{}{}{\itshape}{}{\bfseries}{.}{.5em}{#1\thmnote{ #3}}
\theoremstyle{named}

\newcommand{\rT}{\mathrm{T}}

\newcommand{\fg}{\mathfrak{g}}

\newcommand{\fh}{\mathfrak{h}}

\newcommand{\CC}{\mathbb{C}}

\renewcommand{\phi}{\varphi}
\renewcommand{\emptyset}{\varnothing}

\renewcommand{\tilde}[1]{\widetilde{#1}}

\DeclareMathOperator{\im}{im}

\DeclareMathOperator{\Aut}{Aut}

\newcommand{\GL}{GL}

\newcommand{\fsl}{\mathfrak{sl}}
\newcommand{\fgl}{\mathfrak{gl}}

\newcommand{\SYT}{\mathrm{SYT}}
\newcommand{\SSYT}{\mathrm{SSYT}}
\newcommand{\GTP}{\mathrm{GTP}}
\newcommand{\Tab}{\mathrm{Tab}}

\newcommand\iso\cong
\newcommand\into\hookrightarrow
\newcommand\onto\twoheadrightarrow

\newcommand{\nc}{\newcommand}
\nc{\la}{\lambda}
\nc{\Iso}{\mathsf{Iso}}
\nc{\Irr}{\mathsf{Irr}}
\nc{\Id}{\mathrm{Id}}

\usepackage{ytableau}
\usetikzlibrary{arrows}
\ytableausetup{aligntableaux=center, boxsize=0.5cm}

\newcommand{\young}{\ytableaushort}
\newcommand{\jdt}{\mathsf{jdt}}
\newcommand{\Fold}{\mathsf{Fold}}

\newcommand{\sst}{\scriptstyle}

\newcommand{\kk}{\sst k+1}
\newcommand{\kkk}{\sst k+2}

\newcommand{\green}{*(green)}

\begin{document}
\title{On Sch\"utzenberger modules of the cactus group}
\author{Jongmin Lim}
\address{J.~Lim: School of Mathematics and Statistics, University of Sydney, Australia}
\email{jlim5750@uni.sydney.edu.au}
\author{Oded Yacobi}
\address{O.~Yacobi: School of Mathematics and Statistics, University of Sydney, Australia}
\email{oded.yacobi@sydney.edu.au}

\date{}

\maketitle

\begin{abstract}
The cactus group acts on the set of standard Young tableau of a given shape by (partial) Sch\"utzenberger involutions.  It is natural to extend this action to the corresponding Specht module by identifying standard Young tableau with the Kazhdan-Lusztig basis.  We term these representations of the cactus group ``Sch\"utzenberger modules", denoted $S^\lambda_{\mathsf{Sch}}$, and in this paper we investigate their decomposition into irreducible components.  We prove that when $\lambda$ is a hook shape, the cactus group action on $S^\lambda_{\mathsf{Sch}}$ factors through $S_{n-1}$ and the resulting multiplicities are given by Kostka coefficients.  Our proof relies on results of Berenstein and Kirillov and Chmutov, Glick, and Pylyavskyy.
\end{abstract}

\section{Introduction}
\subsection{Background} Let $\fg$ be a reductive complex Lie algebra.  
In Kashiwara's theory of $\fg$-crystals, the cactus group plays a role analogous to that of the braid group in representations of the quantum group $U_q(\fg)$. Indeed just as the $n$-strand braid group acts on $n$-fold tensor products of representations of $U_q(\fg)$ (resulting in a braided category), the cactus group $C_n$ acts on $n$-fold tensor products of crystals (resulting in a coboundary category) \cite{HK06}.  And just as  the type $\fg$ braid group acts on any integrable representation of $U_q(\fg)$, the type $\fg$ cactus group acts on any $\fg$-crystal \cite{HKRW}.  This latter ``internal'' action is our focus.

Before describing our results, we highlight the appearance of the internal cactus group action in several recent  theorems.

Losev construced an action of the cactus group on the Weyl group of $\fg$, and showed that it interacts nicely with Kazhdan-Lusztig cells \cite{LosCacti}.  For $\fg=\fsl_n$, this recovers the external action of the cactus group corresponding to the zero weight space of the $n$-fold tensor product of the standard representation.  

Losev constructs his action by showing that certain wall-crossing functors are perverse equivalences in the sense of Chuang and Rouquier \cite{CRperv}.  This  was recently extended in work of Halacheva, Losev, Licata and the second author \cite{HLLY}.  We show that for any categorical representation of $U_q(\fg)$, the Rickard complexes corresponding to the half-twist are perverse equivalences. From this we obtain the internal cactus group action on any integrable representation.  In \cite{GY, GY2} Gossow and the second author explain how to recover this cactus group action in type A directly from the representation without appealing to categorical techniques.

In a different direction, Halacheva, Kamnitzer, Rybnikov and Weekes study the action of Gaudin algebras on tensor product multiplicity spaces \cite{HKRW}.  Their main tool is a crystal structure on eigenvectors for shift of argument subalgebras, which are a family of commutative algebras acting on irreducible $\fg$ representations.  In particular they show that the internal action of the cactus group controls the monodromy of these eigenvectors.   

In this paper we initiate a study of representations of the cactus group which arise as permutation modules from the internal action on crystals.  
We'll now describe our work in detail.  

\subsection{Our work} We specialise to the case of $\fg=\fsl_n$.  The corresponding cactus group $C_n$ is an infinite group generated by $c_J$, for subintervals $J \subset [1,n]$, subject to the relations in Definition \ref{cactdef}.  It is isomorphic to the orbifold fundamental group of the real locus of the wonderful compactification of $\fh^{reg}/S_n$, where $\fh^{reg}$ is the regular locus in the reflection representation of the symmetric group $S_n$ \cite{DJS03}.   

Let $\lambda$ be a partition of $n$.  The cactus group $C_n$ acts on  $\SYT(\lambda)$, the set of standard Young tableau of shape $\lambda$, where the generator $c_J$ acts by a partial Sch\"utzenberger involution.  Letting $S^\lambda$ denote the Specht module of $S_n$, it is natural to view the $C_n$-action on $\SYT(\lambda)$ as an action on the Kazhdan-Lusztig basis of $S^\lambda$.  We thus obtain a $C_n$-action on $S^\lambda$, which we term the ``Sch\"utzenberger module'', and denote $S^\lambda_{\mathsf{Sch}}$. 

Our main problem, which to our knowledge has not been studied, is the following:

\begin{Problem}\label{prob}
Determine the irreducible constituents of $S^\lambda_{\mathsf{Sch}}$.
\end{Problem}

An obvious obstruction to solving this problem is that we do not have a classification of the finite dimensional irreducible representation of $C_n$ (it is of wild representation type).  Nevertheless, there are  naturally occurring families of irreducible representations of $C_n$ obtained by inflation from symmetric groups.  

Indeed, there is a natural homomorphism $C_n \to S_n$, and this can be generalised to a surjective map $\pi_k:C_n \to S_{k}$, for $1\leq k \leq n$ (cf. Lemma \ref{lem:pik}).  For $\lambda \vdash k$ we let $S^\lambda_{\pi_k}$ be the irreducible $C_n$-module on $S^\lambda$ obtained via pullback by $\pi_k$.

Our main theorem solves Problem \ref{prob} in the case when $\lambda$ is a hook partition.  To $\lambda=(a,1^b)$ a hook partition, we associate a composition of $n-1$ given by $\tilde{\lambda}=(a-1,b)$.  

\begin{Theorem}\label{main:thm}
Let $\lambda \vdash n$ be a hook partition.
We have an isomorphism of $C_n$-modules
\begin{equation}\label{eq:mainthm}
S^\lambda_{\mathsf{Sch}} \cong \bigoplus_{\mu \;\vdash\; n-1} K_{\mu\tilde{\lambda}} S^{\mu}_{\pi_{n-1}}
\end{equation}
where $K_{\mu\tilde{\lambda}}$ are the Kostka numbers, unless $\lambda=(2,1)$. 
\end{Theorem}

 Note that in the outlying case, $S^{(2,1)}_{\mathsf{Sch}}$ is simply the two-dimensional module with basis elements interchanged by $c_{[1,3]}$ ($c_{[1,2]}$ and $c_{[2,3]}$ act trivially).

Our main tools for proving the theorem come from work of Berenstein and Kirillov \cite{BeKi} and Chmutov, Glick, and Pylyavskyy \cite{CGP}.  The former define a group of symmetries of Gelfand-Tsetlin patterns (i.e. semistandard Young tableau), which the latter show is a quotient of the cactus group.  These results allow us to show that in the case of a hook shape, the $C_n$-action on $S^\lambda_{\mathsf{Sch}}$ factors through $S_{n-1}$, and to identify resulting permutation module.



\begin{Acknowledgement}
This work was undertaken for an Honours Thesis by the first author at the University of Sydney in 2020, under the supervision of the second author.  The second author is partially supported by the ARC grant DP180102563. 
\end{Acknowledgement}

\section{Background}

\subsection{Young tableau}

In this section we briefly recall the basic combinatorics of Young tableau.  For more details see \cite{Sagan}.  Let $n \geq 1$.  A \textbf{partition} of $n$, written $\lambda \vdash n$, is a weakly decreasing sequence of nonnegative integers that sum to $n$: 
$$
\lambda=(\lambda_1,\hdots,\lambda_n),\; \lambda_1\geq \cdots \geq \lambda_n\geq0,\; \sum_i\lambda_i=n.
$$
If we drop the weakly decreasing condition, we get the notion of a \textbf{composition} of $n$.

We use \textbf{Young diagrams} to represent partitions and compositions.  A Young diagram for a composition $\mu$ is a finite collection of cells, arranged in left-justified rows, where the $i$-th row length is the $i$-th entry of $\mu$.  

Let $\lambda \vdash n$.  A \textbf{Young tableau} of shape $\lambda$ is a filling of the corresponding Young diagram with positive integers. For example, here is a Young diagram and tableau of shape $\lambda=(5,3,2)$:
$$
\young{12746,318,29}
$$
The Young tableau is \textbf{semistandard} (respectively \textbf{standard}) if the entries are weakly increasing (respectively strictly increasing) along rows, and strictly increasing down columns.  The \textbf{content} of a tableau $T$ of shape $\lambda$ is the composition of $n$, $\mu(T)=(\mu_1,\mu_2,\hdots)$, where $\mu_i$ is the number of $i$'s appearing in $T$.  

Given $\lambda \vdash n$ and $m\geq 1$, we let $\SSYT(\lambda,m)$ denote the set of semistandard Young tableau of shape $\lambda$ and cells filled in with numbers $1,\hdots,m$.  We let $\SYT(\lambda)$ denote the set of standard Young tableau of shape $\lambda$ and cells filled in with the numbers $1,\hdots,n$.

The \textbf{Kostka number} $K_{\lambda \mu}$ is defined equivalently as: the number of $T \in\SSYT(\lambda,n)$ of shape $\lambda$ and content $\mu$, the dimension of the $\mu$ weight space in the irreducible representation of $\fgl_n$ of highest weight $\lambda$, or as the multiplicity of $S^\lambda$ in the permutation module $M^\mu$ (Equation (\ref{eq:kostka})).

Given partitions $\mu,\lambda$ we write $\mu \subseteq \lambda$ if $\mu_i \leq \lambda_i$ for every $i$.  Let $\lambda, \mu$ be two partitions such that $\mu \subseteq \lambda$. The skew-diagram of shape $\lambda/\mu$ is given by removing the boxes of $\mu$ in $\lambda$. A \textbf{skew tableau} is a labelling of these boxes with positive integers.  Here is a skew diagram and tableau of shape $\lambda/\mu$ for $\lambda = (5,5,3)$ and $\mu = (3, 2)$:
$$\young{\none\none\none 31,\none\none213, 311}$$

Similar to above, semistandard tableau on skew shapes are skew-tableau with weakly increasing labels along the rows and strictly increasing labels down the columns. Standard tableau on skew shapes are semistandard tableau whose entries strictly increase along the rows.   

Let $\mu$ be a composition of $n$.  Let $T,T' $ be two diagrams of shape $\mu$ with entries $1,\hdots,n$.  We write $T \sim T'$ if $T$ and $T'$ are row-equivalent, i.e. they have the same entries in each row.  An equivalence class for this relation is a $\mu$\textbf{-tabloid}.  We let $\Tab(\mu)$ be the set of $\mu$-tabloids.  A tabloid can be pictured in a manner similar to tableau.  For example, here is a $(3,4)$-tabloid:
$$\ytableausetup{tabloids}\young{123,4567}$$
%
%
representing the equivalence class of the diagram with entries $1,2,3$ in the first row and $4,5,6,7$ in the second.

Set $[a,b]=\{a,a+1,\hdots,b\}$.  Given a tableau $T$ we let $T|_{[a,b]}$ be the tableau obtain by deleting all cells with entries not in $[a,b]$.

\subsection{The symmetric group} 
Let $n \geq 1$.  Let $S_n$ denote the symmetric group on $\{1, 2, \hdots, n\}$.  Let $s_i\in S_n$ denote the simple transposition swapping $i$ and $i+1$.  Finite dimensional irreducible complex representations of $S_n$ are indexed by partitions $\lambda$ of $n$. The irreducible representation  corresponding to $\lambda \vdash n$ is the \textbf{Specht module} $S^\lambda$.  For instance, $S^{(n)}$ is the trivial representation, $S^{1^n}$ is the sign representation, and $S^{(n-1,1)}$ is the standard representation.  

The Specht module $S^\lambda$ has a remarkable basis indexed by $\SYT(\lambda)$ called the \textbf{Kazhdan-Lusztig basis}, which we denote $\{ b_T \;\mid\; T \in \SYT(\lambda)\}$.  To construct this basis one needs to pass to the Iwahori-Hecke algebra $H_n(q)$ associated to $S_n$.  Kazhdan and Lusztig constructed a canonical basis of the Hecke algebra, which gives rise also to bases of its cell modules.  In type A, these cell modules are the irreducible Specht modules and the specialisation $q \mapsto 1$ leads to the basis $\{b_T\}$.  For more details, see e.g. \cite{KL79, GaMc, Rhoad}.

Given a composition $\mu$ of $n$, let $S_\mu\subseteq S_n$ be the corresponding  parabolic subgroup.  Let $M^\mu$ denote the induced module $\text{Ind}_{S_\mu}^{S_n}(\CC)$ from the trivial representation.  The module $M^\mu$ has a basis indexed by the set of row tabloids $\Tab(\mu)$, where the action is given by permutation of entries.  Kostka numbers  encode the decomposition of $M^\mu$ into Specht modules:
\begin{equation}\label{eq:kostka}
M^\mu \cong \bigoplus_\lambda K_{\lambda \mu}S^\lambda.
\end{equation}

\subsection{The cactus group} 

Given an interval $J=[a,b] \subseteq [1,n]$, let $S_J \subseteq S_n$ be the  subgroup 
of permutations which fix $i \notin J$.  In the notation of the previous section, $S_J$ is the  parabolic subgroup $S_\mu$, where $\mu=(1^{a-1},b-a+1,1^{n-b})$.
Let $w_J \in S_J$ be the longest element, that is $w_J$ ``flips'' the interval $[a,b]$ via $a+i \mapsto b-i$ for $0\leq i \leq b-a$.

%
\begin{Definition}\label{cactdef}
Let $n \geq 2$ be an integer. The \textbf{cactus group} $C_n$ is generated by ${n \choose 2}$ generators $c_J$, indexed by the  intervals $J \in \{[a,b]\ |\ 1 \leq a < b \leq n\}$, subject to the following relations:
$c_{J}^2 = 1$, $c_{J}c_{K} = c_{K}c_{J}$  if $J \cap K = \emptyset$ and,  $c_{J}c_{K} = c_{w_J(K)}c_{J}$ if $K \subseteq J$. 
\end{Definition}

The cactus group is an infinite group, which has its origins in (a) the study of symmetry groups of universal covers of blow-ups of projective hyperplane arrangements \cite{DJS03}, and (b) the study of commutators in the category of crystals for a semisimple Lie algebra \cite{HK06}.  

Note that there is also a slightly different presentation of the cactus group that's often used, where generators are indexed by subdiagrams of the Dynkin diagram of a semisimple Lie algebra $\fg$ \cite{HKRW}.  The cactus group defined above corresponds to type $A_{n-1}$.

Symmetric groups are naturally quotients of the cactus group.  Indeed, we have a map $\pi_n:C_n \to S_n, \; c_J \mapsto w_J$, which is  a surjective group homomorphism since the defining relations of $C_n$ are satisfied also by the elements $w_J \in S_n$.  This map can be generalised as follows

\begin{Lemma}\label{lem:pik}
 For any $1 \leq k \leq n$ the assignment: 
\begin{align*}
c_{[a,b]} \mapsto \begin{cases}w_{[a,b-n+k]} &\text{ if } n-k < b-a, \\ 1 &\text{ otherwise,}
\end{cases}
\end{align*}
defines a surjective group homomorphism $\pi_k:C_n \to S_k$.
\end{Lemma}

\begin{proof}
The third defining relation of $C_n$ is only non-obvious relation to check.  Suppose we have intervals $K=[a,b] \subseteq J \subseteq I$.   We need to show that $\pi_k(c_{J})\pi_k(c_{K}) = \pi_k(c_{w_J(K)})\pi_k(c_{J})$. 

If $n-k \geq b-a$ then $\pi_k(c_{K})=\pi_k(c_{w_J(K)})=1$ so the equation holds.  Otherwise, we have that $n-k < b-a$.  Let $K'=[a, b-n+k]$.  Then a quick calculation shows that $w_J(K)'=w_{J'}(K')$, which proves the desired relation.
\end{proof}

 By inflation we obtain irreducible representations of $C_n$ on $S^\lambda$, for $\lambda \vdash k$ and $1 \leq k \leq n$, which we denote $S^\lambda_{\pi_k}$.   

\begin{Remark}\label{rem:piij}
It's possible to generalise the maps above further.  Given nonnegative numbers $i,j$ such that $i+j < n$, we have a map $\pi_{(i,j)}:C_n \to S_{[1+i,n-j]}\cong S_{n-i-j}$ given by 
\begin{align*}
c_{[a,b]} \mapsto \begin{cases}w_{[a+i,b-j]} &\text{ if } i+j< b-a, \\ 1 &\text{ otherwise,}
\end{cases}
\end{align*}
It's straightforward to check that this satisfies the defining relations of $C_n$.  We recover $\pi_k$ defined in Lemma \ref{lem:pik} as $\pi_{(0,n-k)}$.
\end{Remark}

\subsection{Operations on Young tableau}
In order to construct the Sch\"utzenberger modules, we need to first recall some operations on Young tableau.  For more details see \cite{Sagan}.
\subsubsection{Jeu De Taquin}
The Jeu de Taquin is a map $\jdt$ taking a semistandard skew tableau to a rectified semistandard tableau, which we recall now. Let $T \in \SSYT(\lambda/\mu,n)$. Call a removable box of $\mu$ a \textbf{movable} box of $T$.  Then $\jdt(T)$ is defined as follows:
\begin{enumerate}
\item Choose a movable box $\ytableausetup{notabloids}\young{{*(green)}}$ of $T$. Move this box with the following rules:
\begin{enumerate}
\item If it is adjacent to a box to its east and south, let them be $i$ and $j$ respectively.
$$\young{{*(green)}i,j}$$
If $i < j$, then swap with $\young{i}$
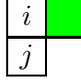
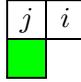
$$\young{i{*(green)},j}$$
Otherwise, swap with $\young{j}$
$$\young{ji,\green}$$
\item If it is adjacent to exactly one box to its east or south, swap it with that box.
\item Repeat this process until it is not adjacent to any boxes to its east or south.
\end{enumerate}
\item Repeat this process with another movable box until there are no movable boxes left.
\end{enumerate}
For example,
$$\jdt\left(\young{\none\none122,\none2445,23}\right) = \young{12225,244\none\none,3\none}$$

The rectification of a skew semistandard tableau via $\jdt$ is independent of the choice of the removable boxes at each iteration.  


\subsubsection{Promotion}
The promotion operation is a map $\partial: \SSYT(\lambda,n) \to \SSYT(\lambda,n)$ defined as follows.
\begin{enumerate}
\item Turn every box labelled 1 to a dummy box.
\item Apply $\jdt$ to the dummy boxes.
\item Reduce every non-dummy box's label by 1.
\item Relabel the dummy boxes to $n$
\end{enumerate}
$$\young{1123,223,44,5}
\quad \longrightarrow \quad\young{\green\green23,223,44,5 }
\quad \longrightarrow\quad\young{2223,34\green,4\green,5}
\quad \longrightarrow\quad\young{1112,235,35,4}$$

\subsubsection{Sch\"utzenberger Involution}
The Sch\"utzenberger Involution is a map $\xi: \SSYT(\lambda,n) \to \SSYT(\lambda,n)$ defined by
$$\xi = \partial_1\circ \partial_2 \circ \cdots \circ \partial_n$$
Where $\partial_k(T)|_{[1, k]} := \partial\left(T|_{[1, k]}\right)$, while leaving $T|_{[k+1, n]}$ constant.

$$\young{1123,223,44,5}
\quad \xrightarrow{\quad\partial_5\quad}\quad\young{1112,235,35,4}
\quad \xrightarrow{\quad\partial_4\quad}\quad\young{1144,225,35,4}
\quad \xrightarrow{\quad\partial_3\quad}\quad\young{1144,235,35,4}$$

The Sch\"utzenberger involution can be shown to be an involution.  Moreover $T$ and $\xi(T)$ are of the same shape, and if $\mu(T) = (\mu_1,\hdots,\mu_n)$ then $\mu(\xi(T)) = (\mu_n,\hdots,\mu_1)$.

\begin{Definition}
Let $\lambda \vdash n$.  For $k = 2, 3, \cdots, n$, define the \textbf{partial Sch\"utzenberger involution} $$\xi_{[1, k]}: \SSYT(\lambda,n) \to \SSYT(\lambda,n)$$ to be the Sch\"utzenberger involution on $T\big|_{[1, k]}$ (where the relabelling is $i \mapsto k+1-i$), while leaving $T\big|_{[k+1, n]}$ constant. 
Let $\xi_{[a, b]} = \xi_{[1, b]} \circ \xi_{[1, b-a+1]} \circ \xi_{[1, b]}$ for $1 \leq a < b \leq n$.
\end{Definition}
\begin{Proposition}\label{cactusproperties}\cite{BeKi}
The operators $\xi_{[a, b]}$ satisfy the following relations:
\end{Proposition}
\begin{enumerate}
\item If $1 \leq i < j <j+1< k < l \leq n$, then $\xi_{[i, j]}\xi_{[k, l]} = \xi_{[k, l]}\xi_{[i, j]}$.
\item For $1 \leq i \leq k < l \leq j \leq n$, we have $\xi_{[i, j]} \xi_{[k, l]}\xi_{[i, j]} = \xi_{[i+j-l, i+j-k]}$.
\end{enumerate}

Define a map $\phi: C_{n} \to \Aut(\SYT(\lambda))$ by $c_J \mapsto \xi_J$ for an interval $J \subseteq [1, n]$.

\begin{Proposition}\label{prop:cacact}\cite{HKRW}
The cactus group $C_n$ acts on the set $\SSYT(\lambda,n)$ via $\phi$, and $\SYT(\lambda)$ is an invariant subset.
\end{Proposition}
\begin{proof}
First, we have
$$\xi_{[1, b]}^2 = 1$$
as we have established that the Sch\"utzenberger involution is an involution. Then we have in general
$$\xi_{[a, b]}^2 = \xi_{[1,b]} \xi_{[1, b-a+1]}^2 \xi_{[1, b]} = 1,$$
and the last two relations follow from Proposition \ref{cactusproperties}.  This shows that $C_n$ acts on $\SSYT(\lambda,n)$.

Note that $\mu(T) = (1, 1, \hdots, 1)$ if and only if $T \in \SYT(\lambda)$. Since $\mu(\xi(T)) = (1, 1, \hdots, 1)$ as well, we have $\xi(T) \in \SYT(\lambda)$. Thus $\xi_{[1, k]}(T) \in \SYT(\lambda)$ and the claim follows.
\end{proof}

\begin{Remark}
The set $\SSYT(\lambda,n)$ naturally carries a $\fsl_n$-crystal structure, isomorphic to the crystal of the irreducible representation of $\fsl_n$ of highest weight $\lambda$ \cite{HoKa}.  The subset $\SYT(\lambda)$ is the weight zero elements of the crystal.  The $C_n$-action described in the above proposition coincides with the internal cactus group action on this crystal.  For more details see \cite{HKRW}, where the internal cactus group action is constructed for any semisimple Lie algebra.  
\end{Remark}

\section{The Sch\"utzenberger modules}
\subsection{Definition and preliminary results}
Let $\lambda \vdash n$.  
Recall the Kazhdan-Lusztig basis $\{b_T\;\mid\; T\in \SYT(\lambda)\}$ of   $S^\lambda$.  We define a  $C_n$ action on $S^\lambda$ using (partial) Sch\"utzenberger involutions: $$c_J \cdot b_T = b_{\xi_J(T)}.$$  We term the resulting representation the \textbf{Sch\"utzenberger module} of $C_n$, and denote it  by $(\rho_\lambda,S^\lambda_{\mathsf{Sch}})$.

Let $v_\lambda = \sum_{T \in \SYT(\lambda)} b_T$ and define the $C_n$-module   $V^\lambda$ by  the decomposition $S^\lambda_{\mathsf{Sch}} = \mathbb{C}v_\lambda \oplus V^\lambda$.  
We begin with some preliminary observations about $V^\lambda$.

\begin{Proposition}\label{smallhook}
Let $n \geq 3$ and set $\lambda = (n-1, 1)$.  Then $V^\lambda$ is an irreducible $C_n$-module.
\end{Proposition}
\begin{proof}
It suffices to show that $\phi: C_{n} \to \Aut(\SYT(\lambda))$ is surjective.
We proceed by induction on $n$. The base case $n = 3$ is trivial.
Define $T_k$ to be the standard tableau with the box $\young{{\sst k+1}}$ in the second row. Then
$\SYT(\lambda) = \{T_1,\hdots,T_{n-1}\}.$

Consider the restriction of $\phi:C_n \to \Aut(\SYT(\lambda))$ to $C_{n-1} \subset C_n$, where we regard $C_{n-1}$ as the subgroup generated by $\{c_{[a, b]}\ |\ 1 \leq a < b \leq n-1\}$. Notice that the image of $C_{n-1}$ does not change the position of the $\young{n}$ box. Thus, $\phi(C_{n-1})$ fixes $\rT_{n-1}$.

On the other hand, if $\mu = (n-2, 1)$, we have a bijection between $\SYT(\mu)$ and $\SYT(\lambda) \setminus \{T_{n-1}\}$ given by appending the box $\young{n}$ to the end of the first row for each $T \in \SYT(\mu)$. This bijection commutes with the $\phi(C_{n-1})$ action as $\phi(C_{n-1})$ disregards $\young{n}$. By the induction hypothesis,
$$\phi(C_{n-1}) \cong \Aut(\SYT(\mu)) \cong \{\sigma \in \Aut(\SYT(\lambda))\ |\ \sigma(T_{n-1}) = T_{n-1}\}.$$
As $\phi(C_n)$ is generated by $\phi(C_{n-1})$ and $\phi(c_{[1, n]})$,
$$\phi(C_n) \cong \langle \{\sigma \in \Aut(\SYT(\lambda))\ |\ \sigma(T_{n-1}) = T_{n-1}\}, \phi(c_{[1, n]}) \rangle \cong \Aut(\SYT(\lambda))$$
as $c_{[1, n]}(T_{n-1}) \neq T_{n-1}$. (Applying the first promotion step shows that the $\young{n}$ box is in the first row for $c_{[1, n]}(T_{n-1})$. In fact, $c_{[1, n]}(T_{n-1}) = T_1$.)
\end{proof}

In general, $V^\lambda$ is not irreducible. Indeed, let $\delta: \SYT(\lambda) \to \SYT(\lambda')$ be the dual map, where the tableau is reflected by the diagonal from northwest to southeast. Here, $\lambda'$ is the dual shape of $\lambda$.
$$\young{1234,56,7} \xrightarrow{\quad \delta \quad} \young{157,26,3,4}$$
The maps $\jdt$ and $\delta$ commute   on standard Young tableaux   (but not in general).
It follows that $\delta$ commutes with the promotion map, and hence the Sch\"utzenberger involution. Thus $\delta$ commutes with the $C_n$ action on standard Young tableau. 

We therefore obtain an isomorphism $\delta:S^\lambda_{\mathsf{Sch}} \to S^{\lambda'}_{\mathsf{Sch}}$.
In particular, for a self-dual shape $\lambda = \lambda'$, we have a non-trivial automorphism $\delta: S^\lambda_{\mathsf{Sch}} \to S^\lambda_{\mathsf{Sch}}$. 
As $\delta$ is an involution, we have an eigenspace decomposition $S^\lambda_{\mathsf{Sch}} = S_+^\lambda \oplus S_-^\lambda$ corresponding to the eigenvalues $\pm1$ of $\delta$: 
\begin{align*}
S_+^\lambda &= span\{b_T + b_{\delta(T)}\ |\ T \in \SYT(\lambda)\} \\
S_-^\lambda &= span\{b_T - b_{\delta(T)}\ |\ T \in \SYT(\lambda)\} 
\end{align*}
Notice that $v_\lambda \in S_+^\lambda$, and hence there exists a submodule $W^\lambda$ such that $S_+^\lambda = \mathbb{C}v_\lambda \oplus W^\lambda$. Thus
$$S^\lambda_{\mathsf{Sch}} = \mathbb{C}v_{\lambda} \oplus W^\lambda \oplus S_-^\lambda,$$
and hence $V^\lambda = W^\lambda \oplus S_-^\lambda$.  Consequently $V^\lambda$ is not irreducible for self-dual $\lambda$.

Theorem \ref{main:thm} generalises Proposition \ref{smallhook} to arbitrary hook-shaped partitions.

\subsection{The Berenstein-Kirillov group}

In order to undertake a more detailed study of Sch\"utzenberger modules  we will utilise Gelfand-Tsetlin patterns and their symmetries, as developed by Berenstein and Kirillov.

Let $n\in \mathbb{N}$ and $\lambda\vdash n$. A \textbf{Gelfand-Tsetlin pattern} with $n$ rows and top row $\lambda$ is a triangular arrangement of nonnegative integers $\{\lambda_{i,j}\}_{1 \leq i \leq j \leq n}$ such that $\lambda_{i, j+1} \geq \lambda_{i, j} \geq \lambda_{i+1, j+1}$ and the top row is $\lambda$. Denote the set of such patterns as $\GTP(\lambda,n)$.

$$
\begin{bmatrix}
\lambda_{1,n} & & \lambda_{2, n} & & \lambda_{3, n} & & \cdots & & \lambda_{n, n}\\
& \lambda_{1, n-1} & & \lambda_{2, n-1} & & \cdots & & \lambda_{n-1, n-1}\\
& & \lambda_{1, n-2} & & \cdots & & \lambda_{n-2, n-2}\\
& & & \ddots & & \reflectbox{$\ddots$}\\
& & & & \lambda_{1, 1}
\end{bmatrix}
$$
Define a map $\Phi: \SSYT(\lambda,n) \to \GTP(\lambda,n)$ as follows. Let $T \in \SSYT(\lambda,n)$. For $1 \leq k \leq n$, as $T$ is semistandard, $T\big|_{[1, k]}$ cannot have  more than $k$ rows. Let the shape of $T\big|_{[1, k]}$ be $(\lambda_{1,k}, \lambda_{2, k}, \cdots, \lambda_{k, k})$. As the notation suggests, set $\Phi(T)$ equal to $\mathcal{T} = \{ \lambda_{i, j}\}_{1 \leq i \leq j \leq n}$.  The following is immediate:
\begin{Proposition}\label{gtpbijection}
$\Phi:\SSYT(\lambda,n) \to \GTP(\lambda,n)$ is a bijection.
\end{Proposition}
  
Bereinstein and Kirillov defined operators acting on $\GTP(\lambda,n)$ as follows.  
Let $\mathcal{T} = \{\lambda_{i, j}\}_{1 \leq i \leq j \leq n} \in \GTP(\lambda,n)$. Define $\tau_k: \GTP(\lambda,n) \to \GTP(\lambda,n)$ for $1 \leq k \leq n-1$ by $\tau_k(\mathcal{T})=\{\widetilde \lambda_{i,j}\}_{1 \leq i \leq j \leq n}$, where
\begin{align*}
a_{i, j} &:= \min(\lambda_{i, j+1}, \lambda_{i-1, j-1})\\
b_{i, j} &:= \max(\lambda_{i,j-1}, \lambda_{i+1, j+1})\\
\widetilde \lambda_{i, j} &:= \lambda_{i, j} \qquad (j \neq k)\\
\widetilde \lambda_{i, k} &:= a_{i, k} + b_{i, k} - \lambda_{i, k}
\end{align*}
For the edge cases we let $a_{1, j} = \lambda_{1, j+1}$ and $b_{j, j} = \lambda_{j+1, j+1}$.

\begin{Proposition}\label{bkdefine}
\cite{BeKi} The operators $\tau_1, \tau_2, \cdots, \tau_{n-1}$ satisfy the following relations.
\begin{align*}
\tau_k^2 &= 1 \quad\quad  1 \leq k \leq n-1\\
\tau_k\tau_l &= \tau_l \tau_k \quad  |k-l| \geq 2\\
(\tau_1\tau_2)^6 &=1\\
(\tau_1q_k)^4 &= 1 \quad\quad k \geq 3
\end{align*}
where $q_k:=  (\tau_1)(\tau_2\tau_1)(\tau_3\tau_2\tau_1) \cdots (\tau_{k} \tau_{k-1} \cdots \tau_1)$. 
\end{Proposition}
It is conjectured by Berenstein and Kirillov that these  generate all relations among the operators $\tau_1, \tau_2, \hdots, \tau_{n-1}$.
\begin{Definition}\label{bkdef}
The \textbf{Berenstein-Kirillov group} $BK_n$ is the group generated by $t_1, t_2, \hdots, t_{n-1}$ with
relations as in Proposition \ref{bkdefine}.
\end{Definition}

By transport of structure via $\Phi$, we have an action of $BK_n$ on $\SSYT(\lambda,n)$.  Following \cite{CGP}, we will describe this explicitly.  Let $T \in \SSYT(\lambda,n)$. Recall that $T\big|_{[k, k+1]}$ is a disjoint union of {\it rectangles} and {\it strips} of the form
$$\young{kk{\none[\cdots]}k,\kk\kk{\none[\cdots]}\kk} \qquad \text{or} \qquad \young{kk{\none[\cdots]}k\kk\kk{\none[\cdots]}\kk}$$
We say a strip is of type $(a, b)$ if it contains $a$ many $\young{k}$-boxes and $b$ many $\young{\kk}$-boxes.

Define $\widetilde{\tau}_k: \SSYT(\lambda,n) \to \SSYT(\lambda,n)$ such that $\widetilde{\tau}_k(T)$ acts on $T\big|_{[k,k+1]}$ by replacing each strip of type $(a, b)$ with a strip of type $(b, a)$, and leaving the rectangles unchanged:
$$\begin{ytableau}
\none & \none & \none & \none & \none & \none &  k &  k &  k &  \sst k+1 &  \sst k+1 &  \sst k+1 \\
\none & \none & \none & \none &  k &  \kk&  \sst k+1 &  \sst k+1\\
\none & k& k & k\\
 \sst k+1 &  \sst k+1 &  \sst k+1
\end{ytableau} \xrightarrow{\quad\widetilde{\tau}_k\quad} \begin{ytableau}
\none & \none & \none & \none & \none & \none &  k &  k &  k & k &  k & \sst k+1 \\
\none & \none & \none & \none &  k &  \kk&  \sst k+1 &  \sst k+1\\
\none & k& k & \kk\\
k &  \sst k+1 &  \sst k+1
\end{ytableau}$$
In the example above, strips of type $(0, 1), (1, 0), (1, 1), (1, 3)$ were swapped with strips of type $(1, 0), (0, 1), (1, 1), (3, 1)$.

We define $\widetilde{\tau}_k(T)$ to be the tableau obtained by replacing $T\big|_{[k,k+1]}$ with $\widetilde{\tau}_k(T\big|_{[k,k+1]})$, and leaving the other boxes unchanged. 
\begin{Lemma}
Let $T \in \SSYT(\lambda,n)$ and $\Phi(T) = \mathcal{T} = \{\lambda_{i, j}\}_{1 \leq i \leq j \leq n} \in \GTP(\lambda,n)$. Recall
$$a_{i, k} = \min(\lambda_{i, k+1}, \lambda_{i-1, k-1}) \qquad b_{i, k} = \max(\lambda_{i+1, k+1}, \lambda_{i, k-1})$$
Then the strip of $T\big|_{[k, k+1]}$ in row $i$ is of type $(\lambda_{i, k} - b_{i, k}, a_{i, k} - \lambda_{i, k})$ starting at column $b_{i, k}+1$.
\end{Lemma}
\begin{proof}
Recall that $\lambda_{i, k}$ corresponds to the number of boxes in row $i$ of $T$ labelled from $1, 2, \cdots, k$.\\
Assume there is no rectangles with its first row in row $i$. This means that every box above $\young{\kk}$ in the $i+1$-th row has a label less than or equal to $k-1$. This is precisely when $\lambda_{i+1, k+1} \leq \lambda_{i, k-1}$, and in this case, the strip indeed starts at column $b_{i, k}+1 = \lambda_{i, k-1}+1$. On the other hand, if there is such a rectangle, then we have $\lambda_{i+1, k+1} > \lambda_{i, k-1}$. This rectangle ends at column $b_{i, k} = \lambda_{i+1, k+1}$, hence the strip starts at column $b_{i, k}+1$ after it as claimed.
$$\young{\none\none kk{\none[\cdots]},\kk{\none[\cdots]}} \qquad \young{\none\none kk{\none[\cdots]},\kk\kk{\none[\cdots]}} \qquad \young{\none\none kk{\none[\cdots]},\kk\kk\kk{\none[\cdots]}}$$
$$\lambda_{i+1, k+1} < \lambda_{i, k-1} \qquad \lambda_{i+1, k+1} = \lambda_{i, k-1} \qquad \lambda_{i+1, k+1} > \lambda_{i, k-1} $$
The proof for $a_{i, k}$ is similar. The strip is indeed of type $(\lambda_{i, k} - b_{i, k}, a_{i, k} - \lambda_{i, k})$ as the boxes labelled $k$ in row $i$ span columns $b_{i,k} +1$ to $\lambda_{i, k}$ by the correspondence given by $\Phi$.
\end{proof}
\begin{Proposition}
For $T \in \SSYT(\lambda,n)$ we have that $t_k\cdot T=\widetilde{\tau}_k(T)$.
%
%
\end{Proposition}
\begin{proof}
It suffices to show that $\Phi \widetilde{\tau}_k =  \tau_k \Phi$.  Notice that $\widetilde{\tau}_k$ does not affect the shape of $T\big|_{[1, j]}$ for all $1\leq j \leq n$ and $j \neq k$. Furthermore, in $T$, each strip in row $i$ spans column $b_{i, k}+1$ to column $a_{i, k}$ and is of type $(\lambda_{i, k} - b_{i, k}, a_{i, k} - \lambda_{i, k})$. Thus in $\widetilde{\tau}_k(T)$, this strip is replaced by a strip of type $(a_{i, k} - \lambda_{i, k}, \lambda_{i, k} - b_{i, k})$.

Let $\mathcal{T}=\Phi(T)$ and set $\tau_k(\mathcal{T}) = \{\widetilde\lambda_{i, j}\}_{1 \leq i \leq j \leq n}$. Let
$$\tilde{a}_{i, k} = \min(\tilde{\lambda}_{i, k+1}, \tilde{\lambda}_{i-1, k-1}) \qquad \tilde{b}_{i, k} = \max(\tilde{\lambda}_{i+1, k+1}, \tilde{\lambda}_{i, k-1})$$

Recall that $\tau_k$ is an operation on $\GTP(\lambda,n)$ which affects only the $k$-th row, hence $\lambda_{i, j} = \widetilde \lambda_{i, j}$ for all $j \neq k$. Thus, $\widetilde a_{i, k} = a_{i, k}$ as follows:
\begin{align*}
\widetilde a_{i, k} &= \min(\widetilde \lambda_{i, k+1}, \widetilde \lambda_{i-1, k-1})\\
&= \min(\lambda_{i, k+1}, \lambda_{i-1, k-1}) = a_{i, k}
\end{align*}
Similarly, we have $\widetilde b_{i, k} = b_{i, k}$. Thus the strip at row $i$ for $\Phi^{-1}(\tau_k(\mathcal{T}))$ also starts and ends at the same column, but is of type $(\widetilde\lambda_{i, k} - b_{i, k}, a_{i, k} - \widetilde\lambda_{i, k})$. However, as $\widetilde \lambda_{i, k} = a_{i, k} + b_{i, k} - \lambda_{i, k}$,
$$(\widetilde\lambda_{i, k} - b_{i, k}, a_{i, k} - \widetilde\lambda_{i, k}) = (a_{i, k} - \lambda_{i, k}, \lambda_{i, k} - b_{i, k})$$
Thus $\widetilde{\tau}_k(T) = \Phi^{-1}(\tau_k(\mathcal{T}))$.
\end{proof}

This proposition implies that for standard Young tableaux, the action of $BK_n$ is particularly easy to describe:

\begin{Corollary}\label{swaps}
Let $T \in \SYT(\lambda)$ for $\lambda \vdash n$. Then $t_k$ swaps the two boxes $\young{k}$ and $\young{\kk}$ if they are not adjacent, otherwise $t_k\cdot T = T$.
\end{Corollary}
\begin{proof}
As $T$ is standard, $T\big|_{[k, k+1]}$ consists of two boxes, which can be non-adjacent, horizontally adjacent, or vertically adjacent:
$$\young{\none k,\kk} \qquad \young{k\kk} \qquad \young{k,\kk}$$
The non-adjacent case is essentially two disjoint strips of type $(1, 0)$ and $(0, 1)$ each. Thus  $\widetilde{\tau}_k$  swaps the two boxes. The vertically adjacent case has no strips, while the horizontally adjacent case is of type $(1, 1)$, which stays constant under $\widetilde{\tau}_k$.
\end{proof}

Consider the elements $p_k, q_k \in BK_n$ defined by
$$p_k := t_kt_{k-1}\cdots t_1 \qquad q_k = p_1p_2\cdots p_k$$
Although we won't use the following theorem of Berenstein and Kirillov, we include a (new) proof since it provides important context for what follows.

\begin{Theorem}\label{actionsame}\cite[Section 2]{BeKi}
The action of $p_k$ and $q_k$ on $\SYT(\lambda,n)$ are equivalent to $\partial_{k+1}$ and $c_{[1, k+1]}$,  the promotion and Sch\"utzenberger involution operations on $T\big|_{[1, k+1]}$.
\end{Theorem}

\begin{proof}
Let $T \in \SSYT(\lambda)$. We first prove that $p_k = \partial_{k+1}$ by induction on $k$.\\
For the base case, $\partial_2$ and $t_1$ act by identity on $\young{12}$ and $\young{1,2}$. Thus, $\partial_2 = t_1 = p_1$.\\
For the inductive case, notice that the $\jdt$ step of $\partial_{k}$ and $\partial_{k+1}$ are identical until the dummy box becomes adjacent to the $\young{\kk}$ box.

{\it Case 1:} If the dummy box is never adjacent to $\young{\kk}$ in the $\jdt$ step of $\partial_{k}$,\\
Then the $\jdt$ step of $\partial_{k}$ and $\partial_{k+1}$ are identical, and they only differ in the relabelling step. For $\partial_{k}$, the dummy box is labelled as $k$, and the $\young{\kk}$ box is kept constant, while every other box's label is reduced by 1. On the other hand, for $\partial_{k+1}$, the dummy box is labelled as $k+1$, while every other box, including $\young{\kk}$, has its label reduced by 1. Furthermore, in both cases, $\young{k}$ and $\young{\kk}$ are not adjacent due to our assumption. Thus, $t_k$ acts by swapping $\young{k}$ and $\young{\kk}$ on $\partial_{k}(T)$ and we have $t_k\partial_{k}(T) = \partial_{k+1}(T)$.

{\it Case 2:} If the dummy box comes adjacent to $\young{\kk}$ in the $\jdt$ step of $\partial{k}$,\\
Then the $\jdt$ step of $\partial_{k+1}$ must have an extra step of swapping the dummy box with $\young{\kk}$. Then after the relabelling steps of $\partial_{k}$ and $\partial_{k+1}$, we have $\partial_{k}(T) = \partial_{k+1}(T)$. Furthermore, by assumption, we have $\young{k}$ and $\young{\kk}$ adjacent, thus $t_k$ acts by identity on $\partial_{k}(T)$.

Thus we have overall $t_k \circ \partial_{k} = \partial_{k+1}$. Hence by induction, $\partial_{k+1} = t_k \circ \partial_{k} = t_k \circ p_{k-1} = p_k$. Then we have by definition
$$c_{[1, k+1]} = \xi_k = \partial_1\partial_2\cdots\partial_k = p_1p_2\cdots p_k = q_k$$
and the result follows.
\end{proof}

We now recall a theorem of Chmutov, Glick, and Pylyavskyy, which identifies the Berenstein-Kirillov group with a quotient of the cactus group.

\begin{Definition}\label{redcactdef}
The \textbf{reduced cactus group} $C^0_n$ is the quotient of $C_n$ by the relations
\begin{align*}
c_i c_{i+1} c_i = c_{i+1} c_i c_{i+1} \tag{C3}
\end{align*}
where $c_i= c_{[i, i+1]}$ for $1 \leq i \leq n-1$.
\end{Definition}

\begin{Remark}
Since $c_i = c_{[1, i+2]}c_{2}c_{[1, i+2]}$ and $c_{i+1} = c_{[1, i+2]}c_{1}c_{[1, i+2]}$
the relations defining the reduced cactus group are conjugates of a single relation, that is, $C^0_n =C_n / \langle (c_{[1,2]}c_{[2,3]})^3 \rangle$.
\end{Remark}

The following is the main result of \cite{CGP}.

\begin{Theorem}\label{bkcg}
There is a group isomorphism $\chi: C^0_n \to BK_n$ given by
$$c_{[1, i]} \mapsto q_{i-1} \qquad 2 \leq i \leq n.$$
\end{Theorem}

\begin{Corollary}\label{cor:factors}
Let $\lambda \vdash n$.
The action of $C_n$ on $\SYT(\lambda)$ factors through $C_n^0$.
\end{Corollary}
\begin{proof}
Let $x = (c_{[1,2]}c_{[2,3]})^3 \in C_n$. Notice that $x$ is a non-identity element. Using $c_{[1,2]}c_{[2,3]}=c_{[1,2]}c_{[1,3]}c_{[1,2]}c_{[1,3]}$ we have 
$$x = (c_{[1,2]}c_{[2,3]})^3  = (c_{[1,2]}c_{[1,3]})^6.$$
By Theorem \ref{actionsame}, the action of $x$ is equivalent to the action of $y = (t_1 (t_1t_2t_1))^6 \in BK_n$. However,
$$y = (t_1(t_1t_2t_1))^6 = (t_2t_1)^6 = 1$$
It follows that $x$ acts by identity on $\SYT(\lambda)$, and the action of $C_n$ on $\SYT(\lambda)$ factors through the projection map $\pi: C_n \to C^0_n$.
\end{proof}

\begin{Remark}
The corollary is a special case of a more general result of Kashiwara \cite[Theorem 7.2.2]{Kash94}, which implies that the internal cactus group action on any normal $\fg$-crystal factors over the reduced cactus group of type $\fg$ (see also \cite[Remark 5.21]{HKRW}).
\end{Remark}


\subsection{The case of a hook shape}\label{subsec:hook}

In this section we will prove our main result, which describes the Sch\"utzenberger modules in the case when $\lambda$ is a hook shape, i.e. $\lambda=(m,1,\hdots,1)$ for some $m$ and some number of $1$s.  For this we make crucial use of the connection between the cactus group and the Berenstein-Kirillov group explained in Theorem \ref{bkcg}.

Recall that $BK_n = \langle t_1, t_2, \cdots, t_{n-1}\rangle$ acts on $\SYT(\lambda)$ for $\lambda \vdash n$. This gives rise to a representation
$$\psi: BK_n \to \GL(S^\lambda_{\mathsf{Sch}}).$$
Recall that $t_k$ acts by swapping $\young{k}$ and $\young{\kk}$ if they are not adjacent, and otherwise does nothing. As $\young{1}$ and $\young{2}$ are always adjacent for standard tableaux, $t_1$ always acts by identity, and thus
$\psi(t_1) = 1$.

The remaining generators $\psi(t_2), \cdots, \psi(t_{n-1})$ satisfy the relations of $BK_n$.  In particular,
$$\psi(t_k)^2 = 1 \qquad \text{and} \qquad \psi(t_k)\psi(t_l) = \psi(t_l)\psi(t_k) \qquad |k-l|\geq 2$$
Assume for the purposes of discussion that for $k\geq 2$, we have
\begin{align*}
(\psi(t_k)\psi(t_{k+1}))^3 = 1 \tag{$\star$}
\end{align*}
This would give a surjective group homomorphism
$$\eta: S_{n-1} \to \im(\psi) \qquad s_k \mapsto \psi(t_{k+1})$$
Since $BK_n \cong C_n^0$ (Theorem \ref{bkcg}), and the action of $C_n$ on $\SYT(\lambda)$ factors through $C_n^0$ (Corollary \ref{cor:factors}), this will allow us to use the representation theory of $S_{n-1}$ to study $S^\lambda_{\mathsf{Sch}}$.  The following lemma describes when this approach is feasible. 

\begin{Lemma}
Let $\lambda \vdash n$.  Then  relation $(\star)$ holds for all $T \in \SYT(\lambda)$ if and only if $\lambda=(2,2)$ or $\lambda$ is a hook shape.
\end{Lemma}
\begin{proof}
As $t_k$ and $t_{k+1}$ act on $T \in \SYT(\lambda)$ depending on how $\young{k}$, $\young{\kk}$, and $\young{\kkk}$ are adjacent, we consider the ways the three boxes can be adjacent.
\begin{enumerate}
\item[{\bf Case 1:}] If the three boxes are all non-adjacent, $(\star)$ is true.  For example:
$$\young{\none\none k,\none\kk,\kkk} \xrightarrow{t_{k+1}} \young{\none\none k,\none \kkk,\kk} \xrightarrow{t_{k}} \young{\none\none \kk,\none \kkk,k} \xrightarrow{t_{k+1}} \young{\none\none \kkk,\none\kk,k}
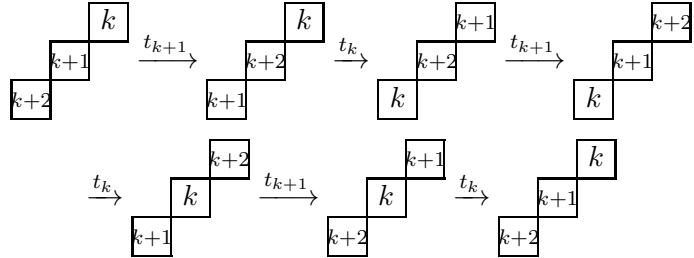
$$
$$\xrightarrow{t_{k}} \young{\none\none \kkk,\none k,\kk} \xrightarrow{t_{k+1}} \young{\none\none \kk,\none k,\kkk} \xrightarrow{t_{k}} \young{\none\none k,\none\kk,\kkk}$$
\item[{\bf Case 2:}] If two boxes are adjacent and one is not adjacent to either, then $(\star)$ is true. For example:
$$\young{\none\none k,\kk\kkk} \xrightarrow{t_{k+1}} \young{\none\none k,\kk\kkk} \xrightarrow{t_{k}} \young{\none\none \kk,k\kkk} \xrightarrow{t_{k+1}} \young{\none\none \kkk,k\kk}
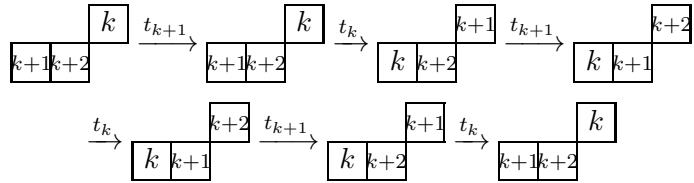
$$
$$\xrightarrow{t_{k}} \young{\none\none \kkk,k\kk} \xrightarrow{t_{k+1}} \young{\none\none \kk,k\kkk} \xrightarrow{t_{k}} \young{\none\none k,\kk\kkk}$$
\item[{\bf Case 3:}] If all three are adjacent in a single row or single column, then $(\star)$ is true.
In this case, $\young{k}$ and $\young{\kk}$ are always adjacent, so $t_k$ always acts by identity. The same is true for $\young{\kk}$ and $\young{\kkk}$, so $t_{k+1}$ also always acts by identity.
\item[{\bf Case 4:}] If all three are adjacent in the following shape, then $(\star)$ is {\bf not} true.
$$\young{k\kk,\kkk}$$
In this case, $\young{k}$ and $\young{\kk}$ are always adjacent while $\young{\kk}$ and $\young{\kkk}$ are always not adjacent. Thus $t_k$ acts by identity while $t_{k+1}$ acts by swapping $\young{\kk}$ and $\young{\kkk}$.
Hence for $T \in \SYT(\lambda)$ with this formation,
$$(\psi(t_k)\psi(t_{k+1}))^3(T) = \psi(t_{k+1})^3(T) = \psi(t_{k+1})(T) \neq T$$
\end{enumerate}
Hence if $\lambda$ is a shape that does not allow the Case 4 configuration for $k\geq 2$ (since we ignore $\psi(t_1)$), then $(\star)$ will hold true. This is exactly when $\lambda = (2,2)$ or $\lambda$ is a hook shape.
\end{proof}

\begin{Remark}
In general for all shapes $\lambda$, we have $(\psi(t_k)\psi(t_{k+1}))^6=1$.
To see that we don't necessarily have $(\psi(t_k)\psi(t_{k+1}))^3=1$, consider $\lambda = (3, 2)$ and the following tableau
$$T = \young{1{3}{4},2{5}}$$
A quick calculation shows:
$$(t_3t_4)^3(T) = \young{135,24}$$
\end{Remark}

\begin{Theorem}\label{thm:psi-iso}
For a hook shape $\lambda$ not of the form $(n),(1^n)$ or $(2,1)$, the map $\eta: S_{n-1} \to \im(\psi)$ is an isomorphism.
\end{Theorem}
\begin{proof}
We have already shown that $\eta$ is surjective. Notice that for $n \geq 4$, $\im(\psi)$ has two distinct non-identity elements, namely $\psi(t_2)$ and $\psi(t_3)$. These are nontrivial because there is  a $T \in \SYT(\lambda)$ such that
$$T\big|_{[1,4]} = \young{124,3} \qquad \text{or} \qquad \young{13,2,4}$$
and indeed $T, \psi(t_2)(T), \psi(t_3)(T)$ are all distinct.

Assume for contradiction $\ker \eta \neq \{1\}$. If $n \neq 5$, the only other normal subgroups of $S_{n-1}$ are $A_{n-1}$, the alternating group, and $S_{n-1}$. In either case, we have $[S_{n-1}:\ker \eta] \leq 2$, hence
$$|S_{n-1}/\ker \eta| \leq 2 < 3 \leq |\im(\psi)|$$
which gives a contradiction. Thus the kernel must be the trivial group for $n \neq 5$.

For $n=5$, the only possible hook shapes are $(4, 1), (3, 1, 1), (2, 1, 1, 1)$.
For the $(4,1)$ case, using the notation from the proof of Proposition \ref{smallhook}, we have that $t_i$ interchanges $T_{i-1}$ and $T_{i}$ for $i=2,3,4$.  Thus $\im(\psi)$ is isomorphic the subgroup of $GL_4(\CC)$ generated by the simple transposition matrices, which is clearly isomorphic to $S_4$.  The $(2,1,1,1)$ case is dual to the $(4,1)$ case.

Let us examine the case for $\lambda = (3, 1,1)$. There are six tableaux in $\SYT(\lambda)$.
$$T_1 = \young{123,4,5} \qquad T_2 = \young{124,3,5} \qquad T_3 = \young{125,3,4}$$
$$T_4 = \young{145,2,3} \qquad T_5 = \young{135,2,4} \qquad T_6 = \young{134,2,5}$$
Hence we can view $x \in \im(\psi)$ as elements in $S_6$ acting on the subscript: $xT_k = T_{x(k)}$.  Then we have:
$$\psi(t_2) = (2\ 6)(3\ 5) \qquad \psi(t_3) = (1\ 2)(4\ 5) \qquad \psi(t_4) = (2\ 3)(5\ 6)$$
The subgroup generated by these elements  has more than six elements. Since every nontrivial normal subgroup of $S_4$ has index at most 6 ($[S_4: K_4] = 6$ where $K_4$ is the Klein four group), we  conclude as above that the map is indeed injective.
\end{proof}
\begin{Corollary}\label{Sfactor}
Let $\lambda \vdash n$ be a hook partition not of the form $(n),(1^n)$ or $(2,1)$.  Then the representation $S^\lambda_{\mathsf{Sch}}$ factors over $S_{n-1}$ as follows:
\[
\begin{tikzcd}
C_n \ar[rr,"\rho_\lambda"] \ar[rd,"\pi_{n-1}"] && GL(S^\lambda_{\mathsf{Sch}}) \\
& S_{n-1} \ar[ur,"\eta"] &
\end{tikzcd}
\]
\end{Corollary}

\begin{proof}
Recall we have the projection map $\pi:C_n \to C_n^0$ and the isomorphism $\chi:C_n^0 \to BK_n$.  
By Theorem \ref{bkcg} and Corollary \ref{cor:factors} we have the diagram:
\[
\begin{tikzcd}
C_n \ar[rr,"\rho_\lambda"] \ar[rd,"\chi\circ\pi"] && GL(S^\lambda_{\mathsf{Sch}}) \\
& BK_n \ar[ur,"\psi"] &
\end{tikzcd}
\]
By Theorem \ref{thm:psi-iso}, if we can extend this diagram with a map from $BK_n$ to $S_{n-1}$, obtaining the desired result.  
\end{proof}

\begin{Definition}
Let $\lambda \vdash n$ be a hook shape.  The boxes in the first row (excluding the first box) are the {\bf arm} of $\lambda$, and the boxes in the first column (excluding the first box) are the {\bf leg} of $\lambda$.
We let $\tilde{\lambda}=(a,b)$ be the two-part composition of $n-1$ formed by the arm and leg of $\lambda$.  

In the example below, $\lambda = (5, 1,1)$ has arm length $4$ and leg length $2$ and $\tilde{\lambda}=(4,2)$.
$$
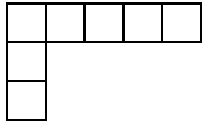
\begin{ytableau}
{} &  &  &   &  \\
\\
\\
\end{ytableau}
$$
\end{Definition}



\begin{Proposition}
Let $\lambda \vdash n$ be a hook shape. Then the $S_{n-1}$ representation $(\eta, S^\lambda_{\mathsf{Sch}})$ is isomorphic to the permutation module $M^{\tilde{\lambda}}$.
\end{Proposition}

\begin{proof}
Set $\tilde{\lambda}=(a,b)$.
We have shown that $S^\lambda_{\mathsf{Sch}} \cong S^{\lambda'}_{\mathsf{Sch}}$ by the dual map, so we can assume without loss of generality that $a \geq b$. 

Since $S^\lambda_{\mathsf{Sch}}$ and $M^{\tilde{\lambda}}$ are both permutation modules of $S_{n-1}$, it suffices to prove that there exists a bijection between the standard bases of each module that commutes with the action of $S_{n-1}$.

Define the operation $\Fold$ given by the following illustration:
$$
\begin{ytableau}
1 & x_1 & x_2 & \none[\cdots] & x_a\\
y_1\\
y_2\\
\none[\vdots]\\
y_b
\end{ytableau} \xrightarrow{\quad \Fold \quad} \ytableausetup{tabloids} \begin{ytableau}
x_1 & x_2 & x_3 & \cdots &x_a\\
y_1 & y_2 & \cdots & y_b
\end{ytableau}$$
Notice that we lose the {\it hinge}, i.e. $\ytableausetup{notabloids}\young{1}$, so the entries in $\Fold(T)$ are now from $\{2, 3, \cdots, n\}$. Thus we subtract 1 from each label, and define a map $F: \SYT(\lambda) \to \Tab(\tilde{\lambda})$ given by $F = (-1) \circ \Fold$.
\begin{figure}[hbt!]
\centering
\begin{tikzpicture}
\node at (-4, 0) (x) {$\young{12345,6,7,8}$};
\node at (0, 0) (y) {$\ytableausetup{tabloids}\young{2345,678}$};
\node at (4, 0) (z) {$\young{1234,567}$};
\draw [->] (x) -- (y) node[midway, above] {$\Fold$};
\draw [->] (y) -- (z) node[midway, above] {$-1$};
\draw [->] (x) to [out=-40, in=-140] (z);
\node at (0, -1.5) {$F$};
\end{tikzpicture}
\end{figure}

We first show that $F: \SYT(\lambda) \to \Tab(\tilde{\lambda})$ is a bijection.
Let $Z = \{2, 3, \cdots, a+b+1\}$. Any $X \subseteq Z$ of cardinality $a$ uniquely determines $T(X) \in \SYT(\lambda)$, where $X$ is the set of numbers in the arm of $T$.
Similarly for $Z' = \{1, 2, \cdots, a+b\}$ any $X' \subseteq Z'$ of cardinality $a$  determines $P(X') \in \Tab(\tilde{\lambda})$, where $X'$ is the set of numbers in the first row.  The bijection is then $F:T(X) \mapsto P(X-1)$.

Next we show that $F$ commutes with $S_{n-1}$.
Recall that $s_i \in S_{n-1}$ acts via $t_{i+1}\in BK_n$ on $\SYT(\lambda)$, which swaps $\ytableausetup{notabloids}\young{{\scriptstyle i+1}}$ and $\young{{\scriptstyle i+2}}$ if the two boxes are not adjacent in the tableau, and if they are it acts by the identity. 

Notice that for $i \geq 1$, $\young{{\scriptstyle i+1}}$ and $\young{{\scriptstyle i+2}}$, are adjacent in $T \in \SYT(\lambda)$ if and only if $\young{i}$ and $\young{{\scriptstyle i+1}}$ are in the same row in $F(T) \in \Tab(\tilde{\lambda})$. Thus, $s_i$ acts on $T$ trivially if and only if it acts trivially on $F(T)$. Otherwise, the boxes swap. In $T$, the boxes swap from the arm to the leg and vice versa. In $F(T)$, the boxes swap rows. Since the arm maps to the first row and the leg maps to the second row, this shows that $F$ commutes with every transposition.
\end{proof}

We are now ready to prove our main result.

\begin{proof}[Proof of Theorem \ref{main:thm}]
In the setting of the theorem, $\lambda$ is a hook partition not equal to $(2,1)$.
Consider first the case when $\lambda$ is not of the form $(n),(1^n)$.  By the above proposition and Equation \eqref{eq:kostka}, there is an isomorphism of $S_{n-1}$-modules
$$
S^\lambda_{\mathsf{Sch}} \cong \bigoplus_{\mu \;\vdash\; n-1} K_{\mu\tilde{\lambda}} S^{\mu}.
$$
By Corollary \ref{Sfactor} this implies the isomorphism of Equation \eqref{eq:mainthm}. 

The remaining cases are easily dealt with by direct computation.  If $\lambda=(n)$ or $\lambda=(1^n)$)  the Kostka number $K_{\mu \tilde{\lambda}}$ is zero unless $\mu=(n-1)$, in which case it is equal to $1$.  Thus both sides of (\ref{eq:mainthm}) are isomorphic to the trivial $C_n$-module.  
\end{proof}

\bibliography{./monbib}

\def\cprime{$'$}
\begin{bibdiv}
\begin{biblist}

\bib{BeKi}{article}{
      author={Berenstein, Arkady},
      author={Kirillov, A.~N.},
       title={Groups generated by involutions, {G}elfand-{T}setlin patterns and
  combinatorics of young tableaux},
        date={1995},
     journal={Algebra i Analiz},
      number={1},
       pages={92\ndash 152},
}

\bib{CGP}{article}{
      author={Chmutov, Michael},
      author={Glick, Max},
      author={Pylyavskyy, Pavlo},
       title={The {B}erenstein-{K}irillov group and cactus groups},
        date={2020},
     journal={J. Comb. Algebra},
      number={2},
       pages={111\ndash 140},
}

\bib{CRperv}{article}{
      author={Chuang, Joseph},
      author={Rouquier, Rapha{\"e}l},
       title={Perverse {E}quivalences},
        note={{P}reprint available on author webpage},
}

\bib{DJS03}{article}{
      author={Davis, M.},
      author={Januszkiewicz, T.},
      author={Scott, R.},
       title={Fundamental groups of blow-ups},
        date={2003},
        ISSN={0001-8708},
     journal={Adv. Math.},
      volume={177},
      number={1},
       pages={115\ndash 179},
         url={https://doi-org.ezproxy.neu.edu/10.1016/S0001-8708(03)00075-6},
      review={\MR{1985196}},
}

\bib{GaMc}{article}{
      author={Garsia, A.M.},
      author={McLarnan, T.J.},
       title={Relations between {Y}oung's natural and the {K}azhdan-{L}usztig
  representations of ${S}_n$},
        date={1988},
     journal={Adv. in Math.},
       pages={32\ndash 92},
}

\bib{GY2}{unpublished}{
      author={Gossow, Fern},
      author={Yacobi, Oded},
       title={On the action of separable elements on dual canonical bases},
         how={In preparation},
        note={In preparation},
}

\bib{GY}{article}{
      author={Gossow, Fern},
      author={Yacobi, Oded},
       title={On the action of the long cycle on the {K}azhdan-{L}usztig
  basis},
        date={2022},
     journal={S{\'e}minaire Lotharingien de Combinatoire},
      volume={86B.11},
}

\bib{HLLY}{article}{
      author={Halacheva, I.},
      author={Losev, I.},
      author={Licata, A.},
      author={Yacobi, O.},
       title={Categorical braid group actions and cactus group},
     journal={arXiv:2102.05931},
}

\bib{HKRW}{article}{
      author={Halacheva, Iva},
      author={Kamnitzer, Joel},
      author={Rybnikov, Leonid},
      author={Weekes, Alex},
       title={Crystals and monodromy of {B}ethe vectors},
        date={2020},
     journal={Duke Math. J.},
      volume={169},
      number={12},
       pages={2337\ndash 2419},
}

\bib{HK06}{article}{
      author={Henriques, Andre},
      author={Kamnitzer, Joel},
       title={Crystals and coboundary categories},
        date={2006},
     journal={Duke Math. J.},
      number={2},
       pages={191\ndash 216},
}

\bib{HoKa}{book}{
      author={Hong, Jin},
      author={Kang, Seok-Jin},
       title={Introduction to quantum groups and crystal bases},
      series={Graduate Studies in Mathematics},
   publisher={Amer. Math. Soc.},
     address={Providence, R.I.},
        date={2002},
      number={42},
}

\bib{Kash94}{article}{
      author={Kashiwara, Masaki},
       title={Crystal bases of modified quantized enveloping algebras},
        date={1994},
     journal={Duke Math. J.},
      number={2},
       pages={383\ndash 413},
}

\bib{KL79}{article}{
      author={Kazhdan, David},
      author={Lusztig, George},
       title={Representations of {C}oxeter groups and {H}ecke algebras},
        date={1979},
        ISSN={0020-9910},
     journal={Invent. Math.},
      volume={53},
      number={2},
       pages={165\ndash 184},
}

\bib{LosCacti}{article}{
      author={Losev, Ivan},
       title={Cacti and cells},
        date={2019},
     journal={J. Eur. Math. Soc. (JEMS)},
      volume={21},
      number={6},
       pages={1729\ndash 1750},
}

\bib{Rhoad}{article}{
      author={Rhoades, Brendon},
       title={Cyclic sieving, promotion, and representation theory},
        date={2010},
     journal={J. Combin. Theory Ser. A},
      volume={117},
      number={1},
       pages={38\ndash 76},
}

\bib{Sagan}{book}{
      author={Sagan, Bruce},
       title={The symmetric group. {R}epresentations, combinatorial algorithms,
  and symmetric functions.},
     edition={Second edition},
      series={Graduate Texts in Mathematics},
   publisher={Springer-Verlag},
        date={2001},
      number={203},
}

\end{biblist}
\end{bibdiv}
\bibliographystyle{amsalpha}

\end{document}